\documentclass[12pt]{amsart}
\usepackage{amssymb}
\usepackage{fullpage}
\usepackage{bbold}
\usepackage[mathscr]{eucal}
\usepackage[all]{xy}
\SelectTips{eu}{}
\numberwithin{equation}{section}
\newtheorem{lemma}[equation]{Lemma}
\newtheorem{theorem}[equation]{Theorem}

\theoremstyle{definition}

\newtheorem{rk}[equation]{Remark}

\newcommand{\Add}{\mathsf{Add}}

\newcommand{\cD}{\mathcal D}
\newcommand{\cH}{\mathcal H}
\newcommand{\cS}{\mathcal S}
\newcommand{\cU}{\mathcal U}

\newcommand{\bZ}{\mathbb Z}

\newcommand{\IS}{\mathcal{IS}}
\newcommand{\Ext}{\mathsf{Ext}}

\newcommand{\Id}{\mathrm{Id}}

\newcommand{\Hom}{\mathsf{Hom}}

\newcommand{\HHH}{\mathsf{H}}

\renewcommand{\le}{\leqslant}

\renewcommand{\geq}{\geqslant}

\title{Bounded complexes of permutation modules}
\author{David J. Benson}
\address{Institute of Mathematics, Fraser Noble Building, Univeristy
  of Aberdeen, Aberdeen AB24 3UE, United Kingdom}
\author{Jon F. Carlson}
\address{Department of Mathematics, University of Georgia, Athens GA
  30602, USA}
\thanks{The second author was partially supported by
Simons Foundation Grant 054813-01.}

\begin{document}

\begin{abstract}
Let $k$ be a field of characteristic $p > 0$. For $G$ an elementary abelian 
$p$-group, there exist collections of permutation module such that if $C^*$ 
is any exact bounded complex whose terms are sums of copies of modules from the 
collection, then $C^*$ is contractible. A consequence is that if $G$ is any
finite group whose Sylow $p$-subgroups are not cyclic or quaternion,
and if $C^*$ is a  bounded exact complex such that each $C^i$ is 
direct sum of one dimensional modules and projective modules, then $C^*$ 
is contractible. 
\end{abstract}

\maketitle

\section{Introduction}

The study of complexes of permutation modules for finite groups has a long history.
They occur naturally in the study of group actions on 
CW-complexes and manifolds,
see for example 
Bredon~\cite{Bredon:1972a},  
Illman~\cite{Illman:1978a}.
Of particular interest are the complexes arising from collections of subgroups in the
work of Quillen~\cite{Quillen:1978a},
Webb~\cite{Webb:1987b,Webb:1987a} and others.
They also arise in the theory of splendid equivalences between derived
categories of blocks, in the work of Rickard~\cite{Rickard:1996a}.

In this paper we examine bounded exact complexes of
permutation modules, and find conditions which force them to be
contractible. Of course for a cyclic group $G\cong\bZ/p$ in
characteristic $p$ there exist bounded exact complexes that are not
contractible, such as the periodicity complex $0\to k \to kG \to kG \to k \to 0$, so the game
is to find conditions prohibiting examples constructed from these.

Suppose that $\cS$ is a collection of modules over the group algebra $kG$ 
of a finite group $G$ with coefficients in a field $k$ of characteristic $p >0$.
We investigate the question of what collections $\cS$ exist with the 
propery that any bounded exact complex of modules in the additive subcategory
$\Add(\cS)$ is contractible. We show that if $G$ is an elementary abelian 
$p$-group having rank at least two, then there are collections of permutation
modules that satisfy this property. In addition, for any finite group having 
$p$-rank at least two, the collection of all modules that are either projective
or have dimension one has the property. There are 
many more such collections.  Our main theorem is
Theorem~\ref{thm:main}, which gives a sufficient condition on a
collection of subgroups of an elementary abelian group for every
bounded exact complex with these stabilisers to be contractible. It is
easy to see, by inflating and inducing up the periodicity complex for a cyclic 
subquotient, that this condition is also necessary. The proof of the
main theorem involves a spectral sequence argument and some
commutative algebra. As an application of the main theorem, we show
that for a finite group of $p$-rank at least two, every bounded exact
complex of sums of projective modules and one dimensional modules is
contractible, in contrast with the cyclic case discussed above.

A recent paper by Paul Balmer and the first author is something of a 
complement to the results presented here. 
In \cite{Balmer/Benson:2020a}, it is proved that 
every module over the modular group algebra of an 
elementary abelian has a finite resolution by 
permutation modules.  By splicing together left and right resolutions
obtained in this way,
we obtain large supplies of exact complexes of permutation modules
that do not split.\medskip

\noindent
{\bf Acknowledgements.}
Both authors wish to thank Henning Krause and the University of Bielefeld 
for their hospitality and support during a visit 
when the initial ideas for this paper
were developed. The authors would also like to thank the Newton
Institute for Mathematical Sciences for support and hospitality during
the programme ``Groups, representations and applications: new
perspectives,'' when part of the work on this paper was undertaken. This work
was supported by EPSRC grant number EP/R014604/1.


\section{preliminaries} \label{sec:prelim}
Let $k$ be a field of characteristic $p$. Let $G$ be an elementary 
abelian $p$-group of rank $r$. That is, $G \cong C_p^r$  is a direct 
product of $r$ copies of the cyclic group $C_p$ of order $p$.   

Recall that if $p=2$, then the cohomology ring of $G$, 
$\HHH^*(G,k) \cong \Ext^*_{kG}(k,k)$ is a polynomial 
ring 
\[ 
\HHH^*(G,k) = k[x_1, \dots, x_r] 
\]
with every $x_i$ having degree one.
If $p$ is odd, then
\[ 
\HHH^*(G,k)=k[x_1, \dots, x_r] \otimes \Lambda(u_1, \dots, u_r) 
\]
where $\Lambda$ is the exterior algebra generated by elements 
$u_i$ is degree one, and every $x_i$ has degree $2.$

If $M$ is a $kG$-module, then $\HHH^*(G,M)$ is a finitely generated 
module over $\HHH^*(G,k)$. For $E$ a subgroup of $G$, let 
$kG/E$ denote the permutation module on the cosets of $E$ in $G$. 
By Frobenius Reciprocity, or the Eckmann--Shapiro Lemma, we have an
isomorphism of $\HHH^*(G,k)$-modules
$\HHH^*(G, kG/E) \cong \HHH^*(E,k)$, with the action of 
$\HHH^*(G,k)$ on $\HHH^*(E,k)$
given by the restriction map $\HHH^*(G,k) \to \HHH^*(E,k)$.

For the proof of the main theorm of the paper we require some  
technical facts. The first is easily verified by the reader. 

\begin{lemma} \label{lem:ann}
Let $R$ be a 
 ring and $u \in R$. 
\begin{enumerate}
\item If $N \subseteq M$ are $R$-modules, and $uN = 0 = u(M/N)$,
then $u^2M = 0$. 
\item If $M$ and $N$ are $R$-modules, $uM = 0$, and $u$ is regular 
on $N$, then $\Hom_R(M, N) = 0$.
\end{enumerate}
\end{lemma}

\begin{theorem}\label{th:hom-ext}
Let $G$ be an elementary abelian $p$-group of rank $r$. 
Let $s\le r$, and let $M$, $N$
be finite direct sums of copies of $\HHH^*(G,k)$-modules
of the form $\HHH^*(E,k)$ where the $E$'s are subgroups of rank $s$.
Let $M_0 \subseteq M$, $N_0 \subseteq N$ be 
graded $\HHH^*(G,k)$-submodules and let
$X=M/M_0$. Suppose that there are elements $u$ and $v$ in
$\HHH^*(G,k)$ such that $u$ and $v$ annihilate $X$ but $u, v$ is
a regular sequence on $N$.
Then
\begin{enumerate}
\item $\Hom_{\HHH^*(G,k)}(X,N)=0$,
\item $\Ext^1_{\HHH^*(G,k)}(X,N)=0$,
\item
$\Hom_{\HHH^*(G,k)}(M,N) \to \Hom_{\HHH^*(G,k)}(M_0,N)$
is an isomorphism,
\item
$\Hom_{\HHH^*(G,k)}(M_0,N_0) \to \Hom_{\HHH^*(G,k)}(M_0,N)$
is injective, and
\item
$\Hom_{\HHH^*(G,k)}(M_0,N_0)$ is zero
in negative degrees, that is, there are no homomorphisms that lower 
degree.
\end{enumerate}
\end{theorem}

\begin{proof}
To prove (i) and (ii), we use a standard depth argument from commutative
algebra. Choose elements $u,v\in \HHH^*(G,k)$ such that $u$ and
$v$ annihilate $X$, but form a regular sequence on $N$. Consider the
exact sequence
\[
\xymatrix{
0\ar[r] &  N \ar[r]^u & N \ar[r] & N/uN \ar[r] & 0. 
}
\]
From this we get a long exact sequence
\[ 
\xymatrix{
0 \ar[r] & \Hom_{\HHH^*(G,k)}(X,N) \ar[r]^u & \Hom_{\HHH^*(G,k)}(X,N) \ar[r] &
\Hom_{\HHH^*(G,k)}(X,N/uN) 
}
\]
\[
\xymatrix{
{} \ar[r] &  \Ext^1_{\HHH^*(G,k)}(X,N) \ar[r]^u &
\Ext^1_{\HHH^*(G,k)}(X,N) \ar[r] & \dots
}
\]
Since $u$ annihilates $X$, it annihilates both $\Hom_{\HHH^*(G,k)}(X,N)$
and $\Ext^1_{\HHH^*(G,k)}(X,N)$. Thus $\Hom_{\HHH^*(G,k)}(X,N) = 0$.
By a similar agument, because $v$ is regular on $N/uN$, we
have that $\Hom_{\HHH^*(G,k)}(X,N/uN)=0$. From this it follows that 
$\Ext^1_{\HHH^*(G,k)}(X,N)=0$ as asserted. 

From the exact sequence
\[
\xymatrix{
0 \ar[r] &  M_0 \ar[r] &  M  \ar[r] &  X \ar[r] &  0 
}
\]
we obtain an exact sequence
\[ 
\xymatrix@-1.2pc{
0\ar[r] &  \Hom_{\HHH^*(G,k)}(X,N) \ar[r] &  \Hom_{\HHH^*(G,k)}(M,N) \ar[r] &
\Hom_{\HHH^*(G,k)}(M_0,N) \ar[r] &  \Ext^1_{\HHH^*(G,k)}(X,N) \ar[r] & \dots 
}
\]
Thus (iii) follows from (i) and (ii).

Part (iv) follows from the fact that the map $N_0\to N$ is injective.
To prove part (v), we note that $\Hom_{\HHH^*(G,k)}(M,N)$ is zero in
negative degrees, so by (iii) the same is true of
$\Hom_{\HHH^*(G,k)}(M_0,N)$, and then by (iv) the same is true of
$\Hom_{\HHH^*(G,k)}(M_0,N_0)$.
\end{proof}


\section{The main theorem}  \label{sec:main}
Our objective in this section is to prove the main theorem of the paper. 
If $\cS$ is a collection of finitely generated $kG$-modules, then 
$\Add(\cS)$ is the full additive subcategory of the module category 
consisting of all $kG$-modules that are isomorphic to finite direct
sums of copies of objects in $\cS$. 

\begin{theorem} \label{thm:main}
Let $G$ be an elementary abelian $p$-group and $k$ a field of
characteristic $p$. Suppose that $\cH$ is a collection of 
subgroups of $G$ with the property that there is no pair $E$, $F$
of elements in $\cH$ such that
$E \subseteq F$ with index $p$. If
\[
\xymatrix{ 
0 \ar[r]  & C^0 \ar[r] &  C^1 \ar[r] &  \cdots \ar[r] & C^\ell \ar[r] & 0
}
\]
is a bounded exact sequence of $kG$-modules in $\Add(\{ kG/E \mid E
\in \cH\})$ then $C^*$ is contractible.
\end{theorem}

Before beginning the proof we note the following. 

\begin{lemma}\label{le:u,v}
Suppose that $\cH$ is a collection of subgroups of $G$ as in 
Theorem \ref{thm:main}. Let $s$ be the maximum rank of an 
element of $\cH$. Let $\cH'$ be the subset of subgroups in $\cH$ of
rank $s$ and let $\cH''=\cH\setminus\cH'$.
If $s > 1$ then there exist homogeneous elements $u$ and $v$ of $H^*(G,k)$
such that for every $E^\prime\in\cH^\prime$,
$u$ and $v$ form a regular sequence on 
$\HHH^*(E^\prime,k)$, and $u$ and $v$ restrict to zero on every
$\HHH^*(E^{\prime\prime},k)$ for all $E^{\prime\prime} \in \cH^{\prime\prime}$.
\end{lemma}

\begin{proof}
This is an easy exercise in
prime avoidance, using the hypothesis on $\cH$ in the theorem.
\end{proof}

We are now prepared to prove our theorem. 

\begin{proof}[Proof of Theoerm \ref{thm:main}]
Suppose that $C^*$ is an exact complex of permutation
modules as in the statement of Theorem \ref{thm:main}. 
Without loss of generality, we may assume that $C^*$ does
not have a nontrivial direct summand that is contractible.
Let 
\[
\xymatrix{
\dots \ar[r] & P_2 \ar[r] & P_1 \ar[r] & P_0 \ar[r] & k \ar[r] & 0
}
\]
be a projective resolution of $k$, the trivial $kG$-module. 
We consider the two spectral sequences whose $E_0$-term
is the double complex
\[ 
E_0^{i,j} = \Hom_{kG}(P_j,C^i). 
\]
If we first take the differential that comes from the complex 
$C^*$, then we get that $E_1^{*,*}$ is identically zero. This is because
$P_j$ is projective and the complex $C^*$ is exact. 

It follows that the spectral sequence obtained by first starting 
with the differential on $P_*$ converges to zero. Indeed, it converges
to zero after a finite number of steps, since there are only a finite
number of columns. The $E_1$ term has the form
\[ 
\xymatrix@-1pc{
E_1^{i,j}= \HHH^j(G,C^i) \quad \ar@{=>}[r] & \qquad 0, 
}
\]
and all differentials on this and subsequent pages 
are homomorphisms of $\HHH^*(G,k)$-modules.
Each $C^i$ is a direct sum of permutation modules on subgroups
$E \in \cH$. The corresponding summand of $\HHH^*(G,C^i)$ is
$\HHH^*(G,kG/E) \cong \HHH^*(E,k)$. Note that the larger the subroup the
smaller is the dimension of the permutation module, and the larger is
the Krull dimension of its cohomology.

Let $s$ be the maximum $p$-rank of any element of $\cH$. 
Let $\cH'$ be the set of subgroups in $\cH$ that have 
rank $s$, and let $\cH''=\cH\setminus\cH'$.
We decompose each $C^i$ as $D^i\oplus U^i$, where
$D^i$ is a direct sum of permutation 
modules $k(G/E)$ with $E \in \cH^\prime$ having rank $s$,
and where $U^i$ is a direct sum of permutation
modules on cosets of subgroups in $\cH''$. By Lemma~\ref{le:u,v},
there are homogeneous  elements $u, v\in \HHH^*(G,k)$ 
which form a regular sequence
on every $\HHH^*(G,D^i)$ but annihilate all $\HHH^*(G,U^i)$.
In particular, since $u$ is regular on $\HHH^*(G,D^i)$
but annihilates $\HHH^*(G,U^i)$, 
there are no nonzero $\HHH^*(G,k)$-module homomorphisms
from $\HHH^*(G,U^i)$ to $\HHH^*(G,D^{i+1})$. Hence, 
$d_1(\HHH^*(G, U^*)) \subseteq \HHH^*(G, U^*)$ and we obtain a short
exact sequence of complexes
\[ 
\xymatrix{
0 \ar[r] & \HHH^*(G,U^*) \ar[r] & E_1^{*,*} \ar[r]^{\vartheta_1 \quad} 
& \HHH^*(G,D^*) \ar[r] & 0. 
}
\]

Our aim is to prove that if $C^*$ is not a contractible complex, then 
the above spectral sequence cannot converge to zero, thus giving us a 
contradiction. Again, we are assuming that $C^*$ has no contractible 
direct summands. We proceed by induction on the pages of the spectral 
sequence. The induction statement is the following.\medskip

\noindent
{\bf Induction Statement for page $n$.}\quad
There exist homogeneous elements $u_n, v_n$ in $\HHH^*(G,k)$ and, 
for every $i,j$, there is a $k$-subspace 
$\cU_n^{i,j} \subseteq E_n^{i,j}$ 
such that the following hold. 
\begin{itemize}
\item[$\IS(n,1)$]: \quad  $\cU_n^{i,*}$ is an $\HHH^*(G,k)$-submodule of
$E_n^{i,*}$. Let $\cD_n^{i,j} = E_n^{i,j}/\cU_n^{i,j}$.
\item[$\IS(n,2)$]: \quad $u_n, v_n$ form a regular sequence on $\HHH^*(G,D^*)$. 
\item[$\IS(n,3)$]: \quad $u_n, v_n$ annihilate $\cU_n^{*,*}$.
\item[$\IS(n,4)$]: \quad There are natural inclusions $\cD_n^{i,j} \subseteq  
\cD_{n-1}^{i,j} \subseteq \dots \subseteq \cD_1^{i,j} = \HHH^*(G,D^*)$ 
as modules over $\HHH^*(G,k)$. Moreover, $\HHH^*(G,D^*)/\cD_n^{*,*}$ 
is annihilated by both $u_n$ and $v_n$. 
\item[$\IS(n,5)$]: \quad  $\cU_n^{*,*}$ is a subcomplex of $(E_n^{*,*},d_n)$ so that 
\[
\xymatrix{ 
0 \ar[r] & \cU_n^{*,*} \ar[r] & E_n^{*,*} \ar[r]^{\vartheta_n} &
\cD_n^{*,*} \ar[r] & 0
}
\]
is an exact sequence of complexes of $\HHH^*(G,k)$-modules.
\item[$\IS(n,6)$]: \quad  $d_n(E_n^{i,j})  \subseteq \cU_n^{i+n,i-n+1}$ 
for all $i$ and $j$. 
\end{itemize} 
\vskip.15in

We begin the induction with $n=1$. Let $\cU_1^{i,j} = \HHH^j(G, U_i)$,
$\cD_1^{i,j} = \HHH^j(G, D_i)$, $u_1 = u,$ $v_1 = v$. Then conditions
$\IS(1,1)$ and $\IS(1,5)$ have already been proved. 
Conditions $\IS(1,2)$ and $\IS(1,3)$ 
are true by the choice of $u$ and $v$. Condition $\IS(1,4)$ is 
obvious. The only remaining task is to verify $\IS(1,6)$ which is easily
seen to be equivalent to the condition that the induced map 
\[
\xymatrix{
\widehat{d}_1\colon \cD_1^{i,*} \cong \HHH^*(G, D^i) \ar[r] & \cD_1^{i+1,*} 
\cong \HHH^*(G, D^{i+1})
}
\]
is the zero map for all $i = 0, \dots, \ell-1$. This is a degree zero map of 
$\HHH^*(G,k)$-modules induced by the coboundary map on $C^*$. But note that 
if $E, F \in \cH^\prime$ and $E \ne F$, then there are no 
nonzero homomorphisms from $\HHH^*(G,k(G/E))$ to $\HHH^*(G,k(G/F))$.
Consequently, the only way that $\widehat{d}_1$ could be nonzero 
is if, for some $i$, $C^i$ and $C^{i+1}$ both had direct summands 
isomorphic to $k(G/E)$ for some $E \in \cH^\prime$,
 and the  composition of the injection
of one followed by the coboundary followed by the projection onto the
other is an isomorpism. However, such a situation does not occur because it 
violates our assumption that the complex  
$C^*$ has no contractible direct summands. 
Thus, the induction statement is true in the case that $n=1$. 

Now we assume for some $n \geq 1$ that the Induction Statement for 
page $n$ is true. We need to show that it holds also for page $n+1$. 
Let $u_{n+1} = u_n^2$
and $v_{n+1} = v_n^2$. Note that $\IS(n+1,2)$ is automatic from $\IS(n,2)$. 
Recall that $E_{n+1}^{i,j}$ is the quotient of $K^{i,j}$ which 
is the kernel of $d_n$ on $E_{n}^{i,j}$ by the submodule 
$d_n(E_{n}^{i-n,j+n-1})$.  Let $\cD_{n+1}^{i,j} = \vartheta_n(K^{i,j})$, the 
image in $\cD_n^{i,j}$ of the kernel of $d_n$ on $E_n^{i,j}$. 
By $\IS(n,6)$, we have that 
$\vartheta_n(d_n(E_{n}^{i-n,j+n-1})) = 0$. Hence there is a well 
defined homomorphism $\vartheta_{n+1}$, and an exact sequence 
\[ 
\xymatrix{
0 \ar[r] & \cU_{n+1}^{*,*} \ar[r] & E_{n+1}^{*,*} \ar[r]^{\vartheta_{n+1}} &
\cD_{n+1}^{*,*} \ar[r] & 0.
}
\]
That is, we \emph{defined} $\cU_{n+1}^{*,*}$ to be the kernel 
$\vartheta_{n+1}$, which is the map induced by $\vartheta_n$.

Because $\vartheta_n$ is an 
$\HHH^*(G,k)$-homomorphism and both the kernel and image of $d_n$ are
$\HHH^*(G,k)$-submodules, we have that $\vartheta_{n+1}$ is also 
an $\HHH^*(G,k)$-homomorphism. This proves $\IS(n+1,1)$.  
Condition $\IS(n+1,3)$ is a consequence of the fact that
$\cU_{n+1}^{*,*}$ is a quotient of $\cU_{n}^{*,*}$ by a submodule 
thereof. So it is annihilated by $u_n$ and $v_n$. For condition 
$\IS(n+1,4)$, notice that if $x \in \cD_n^{i,j}$, then $x = \vartheta_n(y)$
for some $y$ in $E_n^{i,j}$. By $\IS(n,6)$ and $\IS(n,2)$, 
$d_n(y) \in \cU_n^{i+n, j-n+1}$ is annihilated by $u_n$ and $v_n$.
That is, $u_ny$ and $v_ny$ are in the kernel of $d_n$. 
So $\vartheta_n(u_ny) = u_nx$ and  $\vartheta_n(v_ny) = v_nx$ are 
in $\cD_{n+1}^{i,j}$. Thus, $\cD_n^{*,*}/\cD_{n+1}^{*,*}$ is 
annihilated by $u_n$ and $v_n$. It follows from 
$\IS(n,4)$ and Lemma \ref{lem:ann}, that $\HHH^*(G,D^*)$
is annihilated by $u_{n+1}$ and $v_{n+1}$. 

We can see that there are no nonzero homomorphisms from $\cU_{n+1}^{*,*}$
to $\cD_{n+1}^{*,*}$. The reason is that by $\IS(n+1,2)$ and $\IS(n+1,4)$,
$u_{n+1}$ is a regular element on $\cD_{n+1}^{*,*}$, while it annihilates
$\cU_{n+1}^{*,*}$ by $\IS(n+1,3)$ and Lemma \ref{lem:ann}. 
This proves $\IS(n+1,5)$.

Finally, we observe that $d_{n+1}(\cD_{n+1}^{*,*}) \subseteq 
\cU_{n+1}^{*,*}$. The reason for this is that $\cD_{n+1}^{*,*}$ is an
$\HHH^*(G,k)$-submodule of $\HHH^*(G, D^*)$, and the homomorphism
\[
\xymatrix{
\vartheta_{n+1}\, d_{n+1}\colon \cD_{n+1}^{*,*} \ar[r] & \cD_{n+1}^{*,*},
}
\]
which lowers the degrees, is the zero map by Theorem \ref{th:hom-ext}(5).
This implies that $d_{n+1}(E_{n+1}^{*,*}) \subseteq \cU_{n+1}^{*,*}$.
Consequently, Condition $\IS(n+1,6)$ holds. 

Thus, we have shown that the Induction Statement for page $n$ of the 
spectral sequence implies that of page $n+1$. Hence, the statement holds
for all pages. Because the complex $C^*$ has only $\ell+1$ nonzero terms, 
the spectral sequence has only $\ell+1$ nonzero columns, and it must stop
after $\ell+2$ steps. That is, $E_{\ell+2}^{*,*} = E_{\infty}^{*,*} = \{0\}$.
However, this is a contradiction. By $\IS(\ell+2,4)$, $\cD_{\ell+2}^{*,*}$
is an $\HHH^*(G,k)$-submodule of $\HHH^*(G, D^*)$ such that 
$\HHH^*(G, D^*)/\cD_{\ell+2}^{*,*}$ is annihilated by $u_{\ell+2}$ which is a 
regular element on $\HHH^*(G, D^*)$. Thus $\cD_{\ell+2}^{*,*}$ and also 
$E_{\ell+2}^{*,*}$ cannot be zero. This proves the theorem. 
\end{proof}

\section{An application} \label{sec:app}
We present one easy application of the main theorem. There are numerous
similar variations.  As before assume that $k$ is a field of characteristic
$p>0$. 

\begin{theorem} \label{thm:proj+k}
Suppose that $G$ is a finite group having an elementary
abelian subgroup $E$ of $p$-rank 2. Let $\cS$ be the collection consisting
of all indecomposable projective $kG$-modules and all one dimensional
$kG$-modules. If
\[
\xymatrix{
0 \ar[r]  & C^0 \ar[r]^\partial &  C^1 \ar[r] &  \cdots \ar[r] & 
C^{\ell-1} \ar[r]^{\partial_{\ell-1}} & C^\ell \ar[r] & 0
}
\]
is a bounded exact sequence of $kE$-modules in $\Add(\cS)$, then 
$C^*$ is contractible.
\end{theorem}

\begin{proof}
Without loss of generality we may assume that $C^*$ has no nonzero 
direct summands other than itself. In particular, this means that 
$C^\ell$ is a sum of one dimensional modules. That is, if $C^\ell$ has a 
submodule $P$ that is projective, then $P$ is a direct summand 
of $C^\ell$ and $\partial_{\ell-1}$ followed by the projection on to 
$P$ splits. Thus, the complex $0 \to P \to P \to 0$ is a direct summand
of $C^*$, violating our assumption. 

Next, we notice that the theorem is true if it holds in the case that 
$G$ is a $p$-group. For suppose that $Q$ is a Sylow $p$-subgroup of $G$.
Assume that the restriction of $C^*$ to $Q$ is contractible. Then there 
is a $kQ$-homomorphism $\theta\colon C^{\ell} \to C^{\ell-1}$ such that 
$\partial_{\ell-1}\,\theta = \Id_{C^\ell}$. Let $\psi$ be the map 
$\psi = (1/\vert G: Q \vert)\sum_{x \in G/Q} 
x\theta x^{-1}\colon C^\ell \to C^{\ell-1}$. Then $\psi$ 
is a $kG$-homomorphism that splits 
$\partial_{\ell-1}$. 

We assume now that $G$ is a $p$-group and that each $C^i$ is a direct
sum of trivial modules and a projective module. By Theorem~\ref{thm:main},
the restriction of $C^*$ to the elementary abelian subgroup $E$ is 
contractible. Hence, there is $kE$-homomorphism $\theta\colon C^\ell \to C^{\ell-1}$
such that $\partial_{\ell-1}\,\theta = \Id_{C^\ell}$. Now write, $C^{\ell-1} = M 
\oplus P$ where $M \cong k^n$ is a sum of trivial $kG$-modules and $P$
is a projective module. Let $\rho_P\colon C^{\ell-1} \to P$ be the projection, 
and $\iota_P$ be the inclusion of $P$ into $C^{\ell-1}$.
Let $\rho_M$ and $\iota_M$ be the same for $M$ so that 
$\iota_M\,\rho_M + \iota_P\,\rho_P = \Id_{C^{\ell-1}}$.
Notice that $\partial_{\ell-1}\,\iota_P\,\rho_P\,\theta$ is the zero map. 
The reason is that the image of $\theta$ is in the 
space of fixed points of $E$
on $P$, and since $P$ is a free module,
this is a subset of the radical of $P$. Because $C^\ell$ is a sum
of trivial modules, the radical of $P$ is in the kernel of $\partial_{\ell-1}$. 
It follows that $\partial_{\ell-1}\,\iota_M\,\rho_M\,\theta = \Id_{C^\ell}$. 
That is, $\psi = \iota_M\,\rho_M\,\theta$ splits $\partial_{\ell-1}$. To 
finish the proof we only need to notice that $\psi$ is a $kG$-homomorphism,
since it is a linear map between sums of trivial modules. 
\end{proof}

\begin{rk}
The above theorem is not true for groups of $p$-rank one. 
This means groups whose Sylow $p$-subgroups are cyclic or quaternion. 
Indeed, for such a group
the trivial module is periodic, and so there is a non-contractible
exact complex beginning and ending with the trivial module,
part of a projective resolution of the trivial module,
with all intermediate modules projective.
\end{rk}

\bibliographystyle{amsplain}
\bibliography{../../repcoh.bib}

\end{document}